\documentclass[11pt]{amsart}

\usepackage{amsmath, amsfonts, amssymb,amsthm}
\usepackage{amstext}
\usepackage{mathrsfs}

\usepackage{float,epsf,subfigure}
\usepackage[all,cmtip]{xy}
\usepackage{hyperref}
\usepackage{algorithm}
\usepackage{algorithmic}
\usepackage{mathtools}
\usepackage{float}
\usepackage{bm}
\theoremstyle{plain}
\newtheorem{theorem}{Theorem}[section]

\newtheorem{cor}[theorem]{Corollary}
\newtheorem{def-thm}[theorem]{Definition-Theorem}
\newtheorem{lemma}[theorem]{Lemma}

\theoremstyle{definition}

\def\min{\mathop{\mathrm{min}}}

\begin{document}
\title[Nevanlinna-type theory based on heat diffusion]{Nevanlinna-type theory based on heat diffusion}
\author[X.J. Dong]
{Xianjing Dong}

\address{School of Mathematical Sciences \\ Qufu Normal University \\ Qufu, 273165,  P. R. China}
\email{xjdong@amss.ac.cn}


\date{}
\maketitle \thispagestyle{empty} \setcounter{page}{1}

\begin{abstract}
\noindent We obtain  an analogue  of  Nevanlinna theory of holomorphic mappings from a  complete and stochastically complete K\"ahler manifold into  a complex projective manifold. 
When certain  curvature conditions are imposed, the Nevanlinna-type defect relation based on heat diffusion is derived.
\end{abstract}

\vskip\baselineskip

\setlength\arraycolsep{2pt}
\medskip

\section{Introduction}

 In 2010, Atsuji \cite{at2}  introduced the notions of the so-called Nevanlinna-type functions $\tilde{T}_x(t),$ $\tilde{N}_x(t,a)$ and $\tilde{m}_x(t,a)$ of meromorphic functions on a K\"ahler manifold based on heat diffusion. Using the  approaches and techniques from Brownian motion theory (see, e.g., \cite{bass, itoo, rev}), 
   Atsuji obtained   an analogue  of  the Second Main Theorem
 in Nevanlinna theory:
 \begin{theorem}[Atsuji, \cite{at2}]\label{ah} Let $f$ be a nonconstant meromorphic function on a complete and stochastically complete K\"ahler manifold $M.$ Let  $a_1,\cdots,a_q$ be distinct points in $\mathbb P^1(\mathbb C).$ Assume that $\tilde{T}_x(t)<\infty$ as $0\leq t<\infty,$  $\tilde{T}_x(t)\rightarrow\infty$ as $t\rightarrow\infty$ and $|\tilde{N}_x(t,{\rm{Ric}})|<\infty$ as $0\leq t<\infty.$ Then
 $$\sum_{j=1}^q\tilde{m}_x(t,a_j)+\tilde{N}_1(t,x)\leq 2\tilde{T}_x(t)+2\tilde{N}_x(t,{\rm{Ric}})+O\big{(}\log^+\tilde{T}_x(t)\big{)}+O(1)$$
 holds  for $t\geq0$ outside a set of finite Lebesgue measure. 
 \end{theorem}
To  see how  Brownian motion is applied to the Nevanlinna theory, we refer 
the reader to  \cite{att, at,at1,at2,atsuji} and  refer also to  \cite{carne, Dong}.
  
 In this  paper, we shall develop Atsuji's techniques. In doing so, first of all, we extend the notions of the so-called Nevanlinna-type functions (see Section 2.1). As a generalization of Theorem \ref{ah}, we prove  an analogue of  the Second Main Theorem for holomorphic mappings into complex projective manifolds (see Theorem \ref{nonpositive} below). 
 Furthermore, we also discuss  some defect relations as  analogies  in Nevanlinna theory. 

 In our statements of the following theorems,  $M$ is a complete and stochastically complete K\"ahler manifold.   We first obtain the following logarithmic derivative lemma:
\begin{theorem}\label{log}  Let
$\psi$ be a nonconstant meromorphic function on $M$ such that $\tilde{T}_\psi(t,\omega_{FS})<\infty$ as $0\leq t<\infty.$ Then 
for any $\delta>0,$ there exists a set
$E_{\delta}\subseteq [0, \infty)$ of finite Lebesgue measure such that
\begin{eqnarray*}
   \tilde{m}\left(t,\frac{\|\nabla_M\psi\|}{|\psi|}\right)&\leq& \frac{3+\delta}{2}\log^+ \tilde{T}_\psi(t,\omega_{FS})+O(1)
\end{eqnarray*}
 holds for  $t\not\in E_{\delta}.$ 
 \end{theorem}

The  following result is  the so-called Second Main Theorem:
 \begin{theorem}\label{nonpositive}
 Let $L$ be a positive line bundle over a  complex projective manifold $N$ with $\dim_{\mathbb C}N\leq \dim_{\mathbb C}M.$ 
    Let  $D\in|L|$ be  of simple normal
crossing type.  Let $f:M\rightarrow N$  be a differentiably non-degenerate
holomorphic mapping such that $\tilde{T}_f(t,L)<\infty$ and $|\tilde{T}(t,\mathscr{R}_M)|<\infty$ for $0\leq t<\infty.$
  Then  
  \begin{eqnarray*}
     && \tilde{T}_f(t,L)+\tilde{T}_f(t,K_N)+\tilde{T}(t,\mathscr{R}_M) \\
     &\leq& \tilde{N}_f(t,D)-\tilde N_{f', D}(t, 0)+O\big{(}\log^+ \tilde{T}_f(t,L)\big{)}+O(1)
  \end{eqnarray*}
  holds for  $t\geq0$  outside a set  of finite Lebesgue measure.
\end{theorem}
 In 1972s, Griffiths and co-authors (\cite{gri, gri1, gth})   devised an equi-distribution theory of  holomorphic mappings between algebraic varieties   intersecting simple normal crossing type divisors.
In our investigations, Theorem \ref{nonpositive} 
  considers an analogue of Griffiths'  equi-distribution 
 theory based on heat diffusion.
 
When $M$ has non-negative  Ricci curvature, we prove a defect relation: 
\begin{theorem}\label{zzzz} 
Let $L\rightarrow N$ be a positive line bundle over a  complex projective manifold $N$ with $\dim_{\mathbb C}N\leq \dim_{\mathbb C}M.$ 
    Let $D\in|L|$ be  of   simple normal
crossing type.
  Let $f:M\rightarrow N$  be a differentiably
non-degenerate holomorphic mapping satisfying
\begin{equation}\label{con2}
  \int_1^\infty e^{-\epsilon r^2}dr\int_{B_o(r)}e_{f^*c_1(L,h)}(x)dV(x)<\infty
\end{equation}
 for any $\epsilon>0.$
 Then
$$\tilde{\delta}_f(D)\leq
\overline{\left[\frac{c_1(K^*_N)}{c_1(L)}\right]}.$$
\end{theorem}

\section{First Main Theorem}
\subsection{Dynkin Formula}~

 Let $M$ be a  Riemannian manifold with  
  Laplace-Beltrami operator $\Delta_M.$  A Brownian motion $X_t=(X_t)_{t\geq0}$ in $M$ is a heat diffusion  process  generated by $\Delta_M/2,$ with transition density function $p(t,x,y)$ being  the
minimal positive fundamental solution of  the following heat equation
  $$\frac{\partial}{\partial t}u(t,x)-\frac{1}{2}\Delta_{M}u(t,x)=0.$$
The parabolicity of $M$ means the recurrence of  Brownian motions in $M.$
 We say that  $M$ is  stochastically complete if
$$\int_M p(t,x,y)dV(x)=1$$
holds for all $x\in M.$  By Grigor'yan's criterion (see \cite{yan}),  $M$ is stochastically complete
if $R_M(x)\geq-cr^2(x)-c$ for some constant $c>0,$ where $R_M$ is the pointwise lower bound of  Ricci curvature defined  by
\begin{equation}\label{kappa}
 R_M(x)=\inf_{\xi\in T_xM,  \|\xi\|=1}{\rm{Ric}}_M(\xi, \xi).
\end{equation}
\ \ \ \  Let  $\mathbb P_x$ and  $\mathbb E_x$  be the law and expectation of $X_t$ started at $x$ respectively. The It\^o formula (see \cite{at,NN,itoo}) states  that
$$u(X_t)-u(x)=B\left(\int_0^t\|\nabla_M u\|^2(X_s)ds\right)+\frac{1}{2}\int_0^t\Delta_Mu(X_s)dt, \ \ \mathbb P_x-a.s.$$
for $u\in\mathscr{C}_\flat^2(M),$
where $B_t$ is the  standard  Brownian motion in $\mathbb R$ and $\nabla_M$ is the gradient operator on $M$.
Take expectation on both sides of the equality, it  follows the Dynkin formula
\begin{equation}\label{dynkin}
\mathbb E_x[u(X_t)]-u(x)=\frac{1}{2}\mathbb E_o\left[\int_0^t\Delta_Mu(X_s)ds\right]
\end{equation}
 provided that each term  makes sense.

\subsection{Nevanlinna-type functions}~

 Let $(M, g)$ be a  K\"ahler manifold of complex dimension $m,$ whose Laplace-Beltrami operator is denoted by $\Delta_M$  and  K\"ahler  form  is defined by 
 $$\alpha=\frac{\sqrt{-1}}{\pi}\sum_{i,j=1}^mg_{i\bar{j}}dz_i\wedge d\bar{z}_j.$$ 
\ \ \ \ Let $L$ be a holomorphic line bundle over a complex projective manifold $N.$
 We denote by $H^0(N,L)$  the  space of all holomorphic global sections of $L$ over $N,$ and by $|L|$ the
complete linear system of all effective divisors which are  zero divisors of the sections  in  $H^0(N,L).$ 
Moreover, we use the following standard notations
$$d=\partial+\bar{\partial}, \ \ d^c=\frac{\sqrt{-1}}{4\pi}(\bar{\partial}-\partial) \ \ \text{so that}  \ \ dd^c=\frac{\sqrt{-1}}{2\pi}\partial\bar{\partial}.$$
 \ \ \ \ 
 Let $X_t$ be a Brownian motion in $M$ started from a reference point $o\in M.$ 
  Let
$f: M\rightarrow N$ be a holomorphic mapping satisfying  $f(M)\not\subset {\rm{Supp}}D,$ where $D\in |L|$ is a given divisor. 
Equip $L$ with a Hermitian metric $h,$  there  defines the Chern form $c_1(L,h):=-dd^c\log h.$
It's trivial to check that $\Delta_M\log(h\circ f)$ is well  defined. Since $M$ is K\"ahlerian,  then 
\begin{equation}\label{bbzd}
 \Delta_M\log(h\circ f)=-4m\frac{f^*c_1(L,h)\wedge\alpha^{m-1}}{\alpha^m}. 
\end{equation}
\ \ \ \ Consider  a local trivialization covering  $(\{U_\alpha\},\{e_\alpha\})$ of  $(L,h).$ Taking $0\neq s\in H^0(N, L)$ and writing $s=\tilde{s}e_\alpha$  on $U_\alpha$ locally. 
 We also see  $\Delta_M\log|\tilde{s}\circ f|^2$ is well defined on $M.$ By a simple computation,  we obtain 
  \begin{eqnarray*}
  \Delta_M\log|\tilde{s}\circ f|^2&=&4m\frac{dd^c\log|\tilde{s}\circ f|^2\wedge\alpha^{m-1}}{\alpha^m}, \\
  \Delta_M\log \|s\circ f\|^2&=&\Delta_M\log(h\circ f)+\Delta_M\log|\tilde{s}\circ f|^2.
 \end{eqnarray*}
 
  For a   (1,1)-form $\zeta$ on $M,$ we  use the following  convenient symbols
$$e_{\zeta}=2m\frac{\zeta \wedge\alpha^{m-1}}{\alpha^m}.$$
According to (\ref{bbzd}), 
$$e_{f^*c_1(L,h)}=-\frac{1}{2}\Delta_M\log (h\circ f).$$

 Let $\omega$ be a (1,1)-form on $N.$ 
The \emph{characteristic function} of $f$ with respect to $\omega$  is defined by
$$ \tilde{T}_f(t,\omega)=\frac{1}{2}\mathbb E_o\left[\int_0^{t}e_{f^*\omega}(X_s)ds\right].$$
Define  $\tilde{T}_f(t, L):=\tilde{T}_f(t, c_1(L, h)),$ which is well defined 
up to a bounded term since the compactness of $N.$

Let $s_D$ be the canonical section determined by  $D.$ Suppose that $\|s_D\|<1,$ since the compactness of $N$.   The \emph{proximity function} of $f$ with respect to $D$ is defined by
$$\tilde{m}_f(t,D)=\mathbb E_o\left[\log\frac{1}{\|s_D\circ f(X_{t})\|}\right].$$

To define the counting function via Brownian motions, we need to assume that $M$ is stochastically complete.  The \emph{counting function} of $f$ with respect to $D$ is defined by  
\begin{equation}\label{prob}
  \tilde{N}_f(t,D)=\lim_{\lambda\rightarrow\infty}\lambda\mathbb P_o\left(\sup_{0\leq s\leq t}\log\frac{1}{\|s_D\circ f(X_s)\|}>\lambda\right).
\end{equation}

 In what follows, we assume that $L>0.$ 
 In addition, we also assume that  $f$ is differentiably non-degenerate, by which we mean that the Jacobian matrix of $f$  is of  full rank.
We first give  conditions for $\tilde{T}_f(t,L)<\infty$ as $0<t<\infty$ and  $\tilde{T}_f(t,L)\rightarrow\infty$ as $t\rightarrow\infty.$

\begin{lemma}\label{condition1} Let $R_M$ be defined by $(\ref{kappa}).$ Each  of the following conditions ensures that $\tilde{T}_f(t,L)<\infty$ as $0<t<\infty:$

$(a)$ $f$ has finite energy, i.e., $$\int_Me_{f^*c_1(L,h)}(x)dV(x)<\infty;$$

 $(b)$ the energy density function $e_{f^*c_1(L,h)}(x)$ is bounded$;$ 
 
  $(c)$ $R_M(x)\geq-k(r(x))$ for a nondecreasing function $k\geq0$ on $[0,\infty)$ with
  $k(r)/r^2\rightarrow0$ as $r\rightarrow\infty$ and $(\ref{con2})$ is assumed$;$ 
  
  $(d)$ $R_M(x)\geq-k$ for a constant $k\geq0$ with
  $$\int_1^\infty e^{-\epsilon r^2}\sup_{x\in B_o(r)}e_{f^*c_1(L,h)}(x)dr<\infty$$
  for any $\epsilon>0,$ where $B_o(r)$ denotes the geodesic ball centered at $o$ with radius $r$ in $M.$
\end{lemma}
\begin{proof}
$(a)$ and $(b)$ are immediate. $(c)$ is proved using  the estimate of $p(t,o,x)$ due to Li-Yau \cite{Liyau}.
$(d)$ is confirmed by $(c),$ because the boundedness of Ricci curvature implies that ${\rm{Vol}}(B_o(r))$ has at most the exponential growth. The arguments here refer to  \cite{at2}, Proposition 6.
\end{proof}

\begin{lemma}\label{condition} Each of the following conditions ensures that $\tilde{T}_f(t,L)\rightarrow\infty$ as $t\rightarrow\infty:$
 
$(a)$  there exists no nonconstant bounded subharmonic functions on $M.$ In particular,  $M$ is parabolic$;$ 

 $(b)$ ${\rm{Ric}}_M\geq0.$
\end{lemma}
\begin{proof}  Since $L>0,$ we can identify $N$ with an algebraic subvariety of $\mathbb P^k(\mathbb C)$ for some integer $k>0.$
Let $\mathscr H_N$ be the restriction of hyperplane line bundle $\mathscr H$ over $\mathbb P^k(\mathbb C)$ to $N.$  Note that
\begin{equation}\label{wdgg}
  C_1\tilde{T}_f(t,\mathscr H_N)\leq\tilde{T}_f(t,L)\leq C_2\tilde{T}_f(t,\mathscr H_N)
\end{equation}
for some constants $C_1,C_2>0.$  Let $[w_0:\cdots:w_k]$ stand for the homogeneous coordinate system of $\mathbb P^k(\mathbb C).$ Assuming  $w_0\circ f\not\equiv0$ without loss of generality. Then
$$u:=\log(|w_0\circ f|^2+\cdots+|w_k\circ f|^2)$$
is a nonconstant subharmonic function on $M$.  Since $f$ is differentiably non-degenerate, then  $(a)$   implies that
$$\tilde{T}_f(t,\mathscr H_N)=\frac{1}{4}\mathbb E_o\left[\int_0^t\Delta_M u(X_s)ds\right]\rightarrow\infty$$
as $t\rightarrow\infty.$ By (\ref{wdgg}),  we have $(a)$ holds.
$(b)$ follows from  \cite{at2}, Proposition 7 (ii) (see the details of arguments in  \cite{att}).
\end{proof}
We continue to give conditions guaranteeing   $\tilde N_f(t, D)=0$ for $0< t<\infty$ if $f$ omits $D.$ 
Let $u$ be a nonnegative function on $M.$ Set 
$$\tilde N(t, u)=\lim_{\lambda\rightarrow\infty}\lambda\mathbb P_o\left(\sup_{0\leq s\leq t}u(X_s)>\lambda\right).$$
\begin{lemma}[\cite{at2}]\label{leat} Assume that  the Ricci curvature of $M$ satisfies 
$R_M(x)\geq-c^2r^2(x)-c$
for all $x\in M$ and a constant $c>0,$  where $R_M$ is defined by $(\ref{kappa}).$     If $u$ is a 
nonnegative smooth subharmonic function with
$$\liminf_{r\rightarrow\infty}\frac{1}{r^2}\log\int_{B_o(r)}\Delta_Mu(x)dV(x)<\infty,$$
where  $B_o(r)$ is the geodesic ball  centered at $o$ with radius $r$  in $M,$ and if 
$$\mathbb E_o\left[\int_0^t\Delta_M u(X_s)ds\right]<\infty$$
for $0\leq t<\infty,$ then $\tilde N(t, u)=0$ for $0<t<\infty.$
\end{lemma}

\begin{theorem} Let $(L, h)$ be a positive Hermitian line bundle over $N.$  Assume that  the Ricci curvature of $M$ satisfies 
$R_M(x)\geq-c^2r^2(x)-c$
for all $x\in M$ and a constant $c>0,$ where $R_M$ is defined by $(\ref{kappa}).$   Suppose also that $\tilde T_f(t, L)<\infty$ for $0\leq t<\infty$ and  
$$\liminf_{r\rightarrow\infty}\frac{1}{r^2}\log\int_{B_o(r)}e_{f^*c_1(L, h)}dV(x)<\infty.$$
If $f$ omits $D,$ then $\tilde N_f(t, D)=0$ for $0<t<\infty.$
\end{theorem}
\begin{proof}  The  curvature condition means the stochastically completeness of $M$. Let $(\{U_\alpha\},\{e_\alpha\})$ be a local trivialization covering of $(L,h).$ 
 Write $s_D=\tilde{s}_De_\alpha$   on $U_\alpha$ locally. Then 
$$\Delta_M\log\frac{1}{\|s_D\circ f\|^2}=2e_{f^*c_1(L,h)}-\Delta_M\log|\tilde s_D\circ f|^2.$$
If $f$ omits $D,$ one obtains $\Delta_M\log|\tilde s_D\circ f|=0.$
Notice that $c_1(L,h)>0,$ thus    
 $-\log\|s_D\circ f\|$ is subharmonic if $f$ omits $D.$ Using Lemma \ref{leat}, we have the theorem proved.
\end{proof}

\subsection{First Main Theorem}~

 Assume the same notations as before.
\begin{theorem}\label{first} Let  $M$ be a stochastically complete K\"ahler manifold and   
 $L$  be a  positive line bundle over a  complex projective manifold $N.$    Let $f:M\rightarrow N$ be a holomorphic mapping with  $f(M)\not\subset {\rm{Supp}}(D),$ where  $D\in|L|$ is a given divisor.  If $\tilde T_f(t, L)<\infty$ as $0\leq t<\infty,$ then
$$\tilde{T}_f(t,L)=\tilde{m}_f(t,D)+\tilde{N}_f(t,D)+O(1).$$
\end{theorem}
\begin{proof} Equip  $L$ with  a Hermitian metric $h$ such that $\omega:=c_1(L,h)>0.$ Set
\begin{equation*}\label{stop}
T_\lambda=\inf\left\{t>0: \sup_{0\leq s\leq t}\log\frac{1}{\|s_D\circ f(X_s)\|}>\lambda\right\}.
\end{equation*}
Let $(\{U_\alpha\},\{e_\alpha\})$ be a local trivialization covering of $(L,h)$.
 Write
$s_D=\tilde{s}_De_\alpha$ locally on $U_\alpha.$
 Then
\begin{equation}\label{fb}
  \log\|s_D\circ f\|^2=\log|\tilde{s}_D\circ f|^2+\log(h\circ f).
\end{equation}
Apply Dynkin formula 
 to $\log\|s_D\circ f\|^{-1},$ we get
\begin{eqnarray}\label{cvc}
   && \mathbb E_o\left[\log\frac{1}{\|s_D\circ f(X_{t\wedge T_{\lambda}})\|}\right] \nonumber \\
&=& \frac{1}{2}\mathbb E_o\left[\int_0^{t\wedge T_\lambda}\Delta_M
\log\frac{1}{\|s_D\circ f(X_{s})\|}ds\right]+\log\frac{1}{\|s_D\circ f(o)\|},
\end{eqnarray}
where $t\wedge T_\lambda=\min\{t, T_\lambda\}.$ Since $\log\|s_D\circ f(X_{s})\|^{-1}$ has no
singularities as $0\leq s\leq T_\lambda$ due to the definition of $T_\lambda,$ it concludes by (\ref{fb}) that
$$\Delta_M
\log\frac{1}{\|s_D\circ f(X_{s})\|}=-\frac{1}{2}\Delta_M\log(h\circ f(X_{s}))$$
as $0\leq s\leq T_\lambda$, where we use a fact that $\log|\tilde{s}_D\circ f|$ is harmonic on $M\setminus f^{-1}(D).$
Hence, (\ref{cvc}) turns to
\begin{eqnarray*}
  \mathbb E_o\left[\log\frac{1}{\|s_D\circ f(X_{t\wedge T_{\lambda}})\|}\right]
&=&
  -\frac{1}{4}\mathbb E_o\left[\int_0^{t\wedge T_\lambda}\Delta_M\log(h\circ f(X_{s}))ds\right]+O(1).
\end{eqnarray*}
Since $f^*\omega=-dd^c\log(h\circ f),$ then 
\begin{equation}\label{gj}
  e_{f^*\omega}=-2m\frac{dd^c\log(h\circ f)\wedge\alpha^{m-1}}{\alpha^m}=-\frac{1}{2}\Delta_M\log(h\circ f).
\end{equation}
By the monotone convergence theorem, we  see from (\ref{gj}) that
\begin{eqnarray}\label{tt}
   -\frac{1}{4}\mathbb E_o\left[\int_0^{t\wedge T_\lambda}\Delta_M\log(h\circ f(X_{s}))ds\right]  
   &=&\frac{1}{2}\mathbb E_o\left[\int_0^{t\wedge T_\lambda}e_{f^*\omega}(X_s)ds\right] \\
   &\rightarrow& \tilde{T}_f(t,L)  \nonumber
\end{eqnarray}
as $\lambda\rightarrow\infty$, where we use a fact that $T_\lambda\rightarrow\infty$ $a.s.$ as $\lambda\rightarrow\infty$ since  $f^{-1}(D)$ is polar.
Write the left hand side of  (\ref{cvc}) as two parts: 
$$\mathrm{I}+\mathrm{II}:=\mathbb E_o\left[\log\frac{1}{\|s_D\circ f(X_{t})\|}: t<T_\lambda\right]+
\mathbb E_o\left[\log\frac{1}{\|s_D\circ f(X_{T_{\lambda}})\|}: T_\lambda\leq t\right].$$
Using the monotone convergence theorem,  then  
\begin{equation}\label{mm}
  \mathrm{I}\rightarrow \tilde{m}_f(r,D)
\end{equation}
as $\lambda\rightarrow\infty.$  Moreover, by the definition of $T_\lambda,$ it is trivial to see that 
\begin{equation}\label{nn}
  \mathrm{II}=
\lambda\mathbb P_o\left(\sup_{0\leq s\leq t}\log\frac{1}{\|s_D\circ f(X_s)\|}>\lambda\right)\rightarrow \tilde{N}_f(t,D)
\end{equation}
as $\lambda\rightarrow\infty.$ Combining  (\ref{tt})-(\ref{nn}),
we have the desired result.
\end{proof}

\section{Second Main Theorem and Defect Relation}

\subsection{Logarithmic Derivative Lemma}~

 Let $(M,g)$ be a complete and stochastically complete K\"ahler manifold  of complex dimension $m$, with  the  K\"ahler  form $\alpha$ and the gradient operator $\nabla_M$  associated to $g.$  Let $X_t$ be the Brownian motion in $M$ with generator $\frac{1}{2}\Delta_M,$
started at a fixed  point $o\in M,$ with transition density function $p(t,o,x).$

\begin{lemma}[Calculus Lemma]\label{calculus} Let $k$ be a non-negative function on $M$ so that
 $\mathbb E_o[k(X_{t})]<\infty$ and $\mathbb E_o[\int_0^{t} k(X_s)ds]<\infty$ for $0\leq t<\infty.$
Then for any $\delta>0,$ there exists a set
$E_{\delta}\subseteq [0, \infty)$ of finite Lebesgue measure such that
$$
\mathbb E_o\big{[}k(X_{t})\big{]}\leq
\left(\mathbb E_o\Big{[}\int_0^{t} k(X_s)ds\Big{]}\right)^{1+\delta}$$
holds for  $t\not\in E_{\delta}.$
\end{lemma}
\begin{proof} Set $\gamma(t):=\mathbb E_o[\int_0^{t} k(X_s)ds]$ and $E_\delta:=\{t\in(0,\infty):\gamma'(t)>\gamma^{1+\delta}(t)\},$  then $\gamma'(t)=\mathbb E_o\big{[}k(X_{t})\big{]}.$ The claim  holds for $k\equiv0.$ If $k\not\equiv0$, then we suppose that $\gamma(1)\neq0$ without loss of generality.
Note that
 $$\int_{E_\delta}dt\leq1+\int_1^\infty\frac{\gamma'(t)}{\gamma^{1+\delta}(t)}dt\leq1+\delta^{-1}\gamma^{-\delta}(1)
 <\infty.$$
 This completes the proof.
\end{proof}
Let $\psi$ be a meromorphic function on $M.$ Define
$$\tilde{m}\left(t,\frac{\|\nabla_M\psi\|}{|\psi|}\right)=\mathbb E_o\left[\log^+\frac{\|\nabla_M\psi\|}{|\psi|}(X_t)\right],$$
where 
$$\|\nabla_M\psi\|^2=2\sum_{i,j=1}^mg^{i\bar j}\frac{\partial\psi}{\partial z_i}\overline{\frac{\partial \psi}{\partial  z_j}},$$
in which $(g^{i\bar{j}})$ is the inverse of $(g_{i\bar{j}}).$ 
 Regarding $\psi$  as a meromorphic mapping into $\mathbb P^1(\mathbb C).$ The characteristic function of $\psi$ with respect to the Fubini-Study form $\omega_{FS}$ on  $\mathbb P^1(\mathbb C)$  is defined by
$$\tilde{T}_\psi(t,\omega_{FS}) = \frac{1}{4}\mathbb E_o\left[\int_0^t\Delta_M\log\big{(}1+|\psi(X_s)|^2\big{)}ds\right].$$
Adopting the spherical distance $\|\cdot,\cdot\|$ on  $\mathbb P^1(\mathbb C).$ The proximity function of $\psi$  with respect to
$a\in \mathbb P^1(\mathbb C)$
is defined by
$$\tilde{m}_\psi(t,a)=\mathbb E_o\left[\log\frac{1}{\|\psi(X_t),a\|}\right].$$
Again,  set
$$\tilde{N}_\psi(t,a)=
\lim_{\lambda\rightarrow\infty}\lambda\mathbb P_o\left(\sup_{0\leq s\leq t}\log\frac{1}{\|f(X_s),a\|}>\lambda\right).$$
Using the similar arguments as in the proof of Theorem \ref{first}, we obtain 
$$\tilde{T}_\psi(t,\omega_{FS})=\tilde{m}_\psi(t,a)+\tilde{N}_\psi(t,a)+O(1).$$

 Define a singular metric
$$\Phi=\frac{1}{|\zeta|^2(1+\log^2|\zeta|)}\frac{\sqrt{-1}}{4\pi^2}d\zeta\wedge d\bar \zeta$$
on $\mathbb P^1(\mathbb C).$
A direct computation gives that
\begin{equation}\label{ada}
  \int_{\mathbb P^1(\mathbb C)}\Phi=1, \ \ 4m\pi\frac{\psi^*\Phi\wedge\alpha^{m-1}}{\alpha^m}=\frac{\|\nabla_M\psi\|^2}{|\psi|^2(1+\log^2|\psi|)}.
\end{equation}
Set 
$$\tilde{T}_\psi(t,\Phi)=\frac{1}{2}\mathbb E_o\left[\int_0^te_{\psi^*\Phi}(X_s)ds\right],\ \ e_{\psi^*\Phi}(x)=2m\frac{\psi^*\Phi\wedge\alpha^{m-1}}{\alpha^m}.$$
According to (\ref{ada}), we obtain  
\begin{equation}\label{ffww}
  \tilde{T}_\psi(t,\Phi)=\frac{1}{4\pi}\mathbb E_o\left[\int_0^t\frac{\|\nabla_M\psi\|^2}{|\psi|^2(1+\log^2|\psi|)}(X_s)ds\right].
\end{equation}

\begin{lemma}\label{999a}  Let
$\psi$ be a nonconstant meromorphic function on $M$ such that $\tilde{T}_\psi(t,\omega_{FS})<\infty$ as $0\leq t<\infty.$  
Then for any $\delta>0,$ there exists a set
$E_{\delta}\subseteq [0, \infty)$ of finite Lebesgue measure such that
\begin{eqnarray*}
  \mathbb E_o\left[\log^+\frac{\|\nabla_M\psi\|^2}{|\psi|^2(1+\log^2|\psi|)}(X_{t})\right]\leq  (1+\delta)\log^+\tilde{T}_{\psi}(t,\omega_{FS})+O(1)
\end{eqnarray*}
holds for  $t\not\in E_{\delta}.$
\end{lemma}
\begin{proof} By Jensen inequality
\begin{eqnarray*}
   \mathbb E_o\left[\log^+\frac{\|\nabla_M\psi\|^2}{|\psi|^2(1+\log^2|\psi|)}(X_{t})\right]   
   &\leq&  \mathbb E_o\left[\log\Big{(}1+\frac{\|\nabla_M\psi\|^2}{|\psi|^2(1+\log^2|\psi|)}(X_{t})\Big{)}\right] \nonumber \\
    &\leq& \log^+\mathbb E_o\left[\frac{\|\nabla_M\psi\|^2}{|\psi|^2(1+\log^2|\psi|)}(X_{t})\right]+O(1). \nonumber
\end{eqnarray*}
Applying  Lemma \ref{calculus} and (\ref{ffww})  to get 
\begin{eqnarray*}
   && \log^+\mathbb E_o\left[\frac{\|\nabla_M\psi\|^2}{|\psi|^2(1+\log^2|\psi|)}(X_{t})\right]  \\
   &\leq& (1+\delta)\log^+\mathbb E_o\left[\int_0^t\frac{\|\nabla_M\psi\|^2}{|\psi|^2(1+\log^2|\psi|)}(X_{s})ds\right] \\
   &\leq&   (1+\delta)\log^+\int_0^tds\int_Mp(s,o,x)\psi^*\Phi\wedge\alpha^{m-1}+O(1) \\
   &\leq&  (1+\delta)\log^+\int_{\mathbb P^1(\mathbb C)}\tilde N_\psi(t, \zeta)\Phi(\zeta)+O(1) \\
   &\leq&  (1+\delta)\log^+\tilde{T}_{\psi}(t,\omega_{FS})+O(1). \nonumber
\end{eqnarray*}
\end{proof}

\noindent\emph{Proof of Theorem $\ref{log}$.}   
 On the one hand,
\begin{eqnarray*}
&& \tilde{m}\left(t,\frac{\|\nabla_M\psi\|}{|\psi|}\right) \\
   &=&  \frac{1}{2}\mathbb E_o\left[\log^+\left(\frac{\|\nabla_M\psi\|^2}{|\psi|^2(1+\log^2|\psi|)}(X_t)\big{(}1+\log^2|\psi(X_t)|\big{)}\right)\right]  \\
   &\leq& \frac{1}{2}\mathbb E_o\left[\log^+\frac{\|\nabla_M\psi\|^2}{|\psi|^2(1+\log^2|\psi|)}(X_t)\right]
   +\frac{1}{2}\mathbb E_o\left[\log^+\big{(}1+\log^2|\psi(X_t)|\big{)}\right] \\
   &\leq& \frac{1}{2}\mathbb E_o\left[\log^+\frac{\|\nabla_M\psi\|^2}{|\psi|^2(1+\log^2|\psi|)}(X_{t})\right]  \\
  && +\mathbb E_o\left[\log\Big{(}1+\log^+|\psi(X_t)|+\log^+\frac{1}{|\psi(X_t)|}\Big{)}\right].
\end{eqnarray*}
In which, Lemma \ref{999a} gives that
\begin{eqnarray*}
\mathbb E_o\left[\log^+\frac{\|\nabla_M\psi\|^2}{|\psi|^2(1+\log^2|\psi|)}(X_{t})\right]
 &\leq& (1+\delta)\log^+\tilde{T}_\psi(t,\omega_{FS})+O(1).
\end{eqnarray*}
Moreover, by Jensen inequality
\begin{eqnarray*}
&&\mathbb E_o\left[\log\Big{(}1+\log^+|\psi(X_t)|+\log^+\frac{1}{|\psi(X_t)|}\Big{)}\right] \\
   &\leq& \log^+\left(\tilde{m}_\psi(t,\infty) +\tilde{m}_\psi(t,0)\right)+O(1)  \\
   &\leq& \log^+\tilde{T}_\psi(t,\omega_{FS})+O(1).
\end{eqnarray*}
Combining the above, we prove the theorem.
\subsection{Second Main Theorem}~

Let $(M, g)$ be a  complete and stochastically complete K\"ahler manifold of complex dimension   $m,$ whose  K\"ahler form  is written as  
 $$\alpha=\frac{\sqrt{-1}}{\pi}\sum_{i,j}g_{i\bar{j}}dz_i\wedge d\bar{z}_j.$$
Then 
 $$\alpha^m=m!\det(g_{i\bar j})\bigwedge_{j=1}^m\frac{\sqrt{-1}}{\pi}dz_j\wedge d\bar z_j.$$
Define the Ricci form $\mathscr{R}_M$ of $M$   by
$$\mathscr{R}_M=-dd^c\log\det(g_{s\bar{t}})=\frac{\sqrt{-1}}{2\pi}\sum_{i,j=1}^mR_{i\bar{j}}dz_i\wedge d\bar{z}_j,
$$ 
 where 
$$R_{i\bar{j}}=-\frac{\partial^2}{\partial z_i\partial \overline{z}_j}\log\det(g_{s\bar{t}}).
$$

A well-known theorem by S. S. Chern asserts that $\mathscr{R}_M$ is a real and closed smooth (1,1)-form, which represents the first Chern class of $M$ in de Rham cohomology group $H^2_{{\rm{DR}}}(M,\mathbb R).$   Let $s_M$ be the scalar curvature of $M,$ then
$$s_M=\sum_{i,j}g^{i\bar j}R_{i\bar j},$$
where
$(g^{i\bar j})$ is the inverse of $(g_{i\bar j}).$ A direct computation yields that 
$$s_M=-\frac{1}{2}\Delta_M\log\det(g_{s\bar t}).$$
\ \ \ \ Let $(L,h)$ be a positive Hermitian  line bundle over a  complex projective  manifold $N$ with $\dim_{\mathbb C}N=n\leq m.$ It defines  a  volume form $\Omega=\wedge^nc_1(L,h)$ on $N.$ Write $D=\sum_{j=1}^qD_j\in |L|$ into  a sum of irreducible components, which is 
of  simple normal crossing type, 
  one  can equip every  $L_{D_j}$  ($1\leq j\leq q$) with a Hermitian
metric such that the induced Hermitian metric on $L=\otimes_{j=1}^qL_{D_j}$ is  $h.$
Taking $s_j\in H^0(N, L_{D_j})$
satisfying $(s_j)=D_j$ and $\|s_j\|<1.$
On $N,$ one can define a singular volume form
\begin{equation}\label{1phi}
  \Phi=\frac{\Omega}{\prod_{j=1}^q\|s_j\|^2}.
\end{equation}
Set
$$f^*\Phi\wedge\alpha^{m-n}=\xi\alpha^m.$$ 
Recall that
$$T_\lambda=\inf\left\{t>0: \sup_{0\leq s\leq t}\log\frac{1}{\|s_D\circ f(X_s)\|}>\lambda\right\}.$$
Introduce
 $$\tilde N_{f', D}(t, 0)=\lim_{\lambda\rightarrow\infty} \mathbb E_o\left[\log^-\frac{f^*\Omega\wedge\alpha^{m-n}}{\alpha^m}(X_{T_\lambda}): T_\lambda\leq t\right]. 
$$
Let   $J_f$ denote the set of  points in $M$ such that $f$ is differentiably degenerate, i.e., the rank  of Jacobian matrix of $f$ is not full. 
Notice that $T_\lambda\rightarrow0$ $a.s.$ as $\lambda\rightarrow\infty$ and the image of  $f$  approaches $D$ infinitely $a.s.$  as $\lambda\rightarrow\infty,$ thus one  sees that $\tilde N_{f', D}(t, 0)$ 
 measures the size of  $J_f\cap f^{-1}(D)$ counting multiplicities.   $\tilde N_{f',D}(t, 0)$ may be divergent unless certain curvature conditions are imposed. 

\begin{lemma}\label{lemma} If $\tilde N_f(t, D)+\mathbb E_o\big{[}\log^+\xi(X_{t})\big{]}<\infty$ for $0\leq t<\infty,$ then 
   \begin{eqnarray*}
&& \tilde T_f(t,L)+\tilde T_f(t, K_N)+\tilde T(t, \mathscr R_M)  \\
&\leq& \tilde N_f(t, D)-\tilde N_{f', D}(t, 0)+\frac{1}{2}\mathbb E_o\big{[}\log\xi(X_{t})\big{]}+O(1).
   \end{eqnarray*}
\end{lemma}
\begin{proof} Since   
$$\Delta_M\log \|s_D\circ f(X_t)\|^2=\Delta_M\log h\circ f(X_t)$$
as $0\leq t\leq T_\lambda,$ then 
  it yields from  Dynkin formula  that 
   \begin{eqnarray}\label{A1}
&& \frac{1}{2}\mathbb E_o\big{[}\log\xi(X_{t\wedge T_\lambda})\big{]}  \\
&=& \frac{1}{4}\mathbb E_o\left[\int_0^{t\wedge T_\lambda}\Delta_M\log\xi(X_s)ds\right]+O(1) \nonumber \\
&=&  \frac{1}{4}\mathbb E_o\left[\int_0^{t\wedge T_\lambda}4m\frac{dd^c[\log\xi]\wedge\alpha^{m-n}}{\alpha^m}(X_s)ds\right]+O(1) \nonumber \\
&=& \tilde T_f(t\wedge T_\lambda, L)+\tilde T_f(t\wedge T_\lambda, K_N)+\tilde T(t\wedge T_\lambda, \mathscr R_M)+O(1). \nonumber
 \end{eqnarray}
On the other hand, 
 \begin{eqnarray*}
 \mathbb E_o\big{[}\log\xi(X_{t\wedge T_\lambda})\big{]} 
&=& \mathbb E_o\big{[}\log\xi(X_{t}): t<T_\lambda\big{]}+ \mathbb E_o\big{[}\log\xi(X_{T_\lambda}): T_\lambda\leq t\big{]} \\
&\leq& \mathbb E_o\big{[}\log^+\xi(X_{t}): t<T_\lambda\big{]}+ \mathbb E_o\big{[}\log\xi(X_{T_\lambda}): T_\lambda\leq t\big{]},
 \end{eqnarray*}
 where 
 \begin{eqnarray*}
&& \mathbb E_o\big{[}\log\xi(X_{T_\lambda}): T_\lambda\leq t\big{]} \\
&=&  \mathbb E_o\left[\log\frac{1}{\|s_D\circ f(X_{T_\lambda})\|^2}: T_\lambda\leq t\right]+ 
\mathbb E_o\left[\log\frac{f^*\Omega\wedge\alpha^{m-n}}{\alpha^m}(X_{T_\lambda}): T_\lambda\leq t\right] \\
&\leq& 2\lambda\mathbb P_o\left(\sup_{0\leq s\leq t}\log\frac{1}{\|s_D\circ f(X_s)\|}>\lambda\right)+
\mathbb E_o\left[\log\frac{f^*\Omega\wedge\alpha^{m-n}}{\alpha^m}(X_{T_\lambda}): T_\lambda\leq t\right].
 \end{eqnarray*}
 Thus, 
  \begin{eqnarray*}
&& \frac{1}{2}\mathbb E_o\big{[}\log\xi(X_{t\wedge T_\lambda})\big{]} \\
&\leq& \lambda\mathbb P_o\left(\sup_{0\leq s\leq t}\log\frac{1}{\|s_D\circ f(X_s)\|}>\lambda\right)
+\frac{1}{2}\mathbb E_o\big{[}\log^+\xi(X_{t}): t<T_\lambda\big{]} \\
&&
+\frac{1}{2}\mathbb E_o\left[\log\frac{f^*\Omega\wedge\alpha^{m-n}}{\alpha^m}(X_{T_\lambda}): T_\lambda\leq t\right]. 
 \end{eqnarray*}
Since  $\|s_D\|<1$ and $T_\lambda\rightarrow\infty$ $a.s.$ as $\lambda\rightarrow\infty,$
  it follows from the monotone convergence theorem that 
   \begin{eqnarray*}
&&\lim_{\lambda\rightarrow\infty} \left[\lambda\mathbb P_o\left(\sup_{0\leq s\leq t}\log\frac{1}{\|s_D\circ f(X_s)\|}>\lambda\right)
+\frac{1}{2}\mathbb E_o\big{[}\log^+\xi(X_{t}): t<T_\lambda\big{]}\right] \\
&=& \tilde N_f(t, D)+\frac{1}{2}\mathbb E_o\big{[}\log^+\xi(X_{t})\big{]}
 \end{eqnarray*}
 and 
 $$\lim_{\lambda\rightarrow\infty} \mathbb E_o\left[\log\frac{f^*\Omega\wedge\alpha^{m-n}}{\alpha^m}(X_{T_\lambda}): T_\lambda\leq t\right]=
 -\tilde N_{f', D}(t, 0). 
$$
Therefore,
  \begin{equation}\label{A2}
   \lim_{\lambda\rightarrow\infty} \frac{1}{2}\mathbb E_o\big{[}\log\xi(X_{t\wedge T_\lambda})\big{]}\leq \tilde N_f(t, D)
   -\tilde N_{f', D}(t, 0)+\frac{1}{2}\mathbb E_o\big{[}\log^+\xi(X_{t})\big{]}. 
  \end{equation}
Combining (\ref{A1}) and (\ref{A2}) with conditions, we get  
 \begin{eqnarray}\label{A3}
&&\lim_{\lambda\rightarrow\infty}\left[\tilde T_f(t\wedge T_\lambda, L)+\tilde T_f(t\wedge T_\lambda, K_N)+\tilde T(t\wedge T_\lambda, \mathscr R_M)\right] \\
&\leq& \tilde N_f(t, D)-\tilde N_{f',D}(t, 0)+\frac{1}{2}\mathbb E_o\big{[}\log^+\xi(X_{t})\big{]}+O(1)<\infty.\nonumber
 \end{eqnarray}
Apply  Lebesgue's control convergence theorem to (\ref{A3}), we have the desired result.
  \end{proof}

\noindent\emph{Proof of  Theorem $\ref{nonpositive}$.}
Follow Ru-Wong's arguments (see \cite{ru}, pp. 231-233; see also \cite{No}), 
there exists  a finite  open covering $\{U_\lambda\}$ of $N$ and rational functions
$w_{\lambda1},\cdots,w_{\lambda n}$ on $N$ for every  $\lambda$ such  that $w_{\lambda1},\cdots,w_{\lambda n}$ are holomorphic on $U_\lambda$
and  
$$ \ \ \ \ \ dw_{\lambda1}\wedge\cdots\wedge dw_{\lambda n}(y)\neq0, \ \  ^\forall y\in U_{\lambda}; $$
$$U_{\lambda}\cap D=\big{\{}w_{\lambda1}\cdots w_{\lambda h_\lambda}=0\big{\}}, \ \ ^\exists h_{\lambda}\leq n.$$
In addition, we  can require  $L_{D_j}|_{U_\lambda}\cong U_\lambda\times \mathbb C$ for $\lambda,j.$ On  $U_\lambda,$ we have
$$\Phi=\frac{e_\lambda}{|w_{\lambda1}|^2\cdots|w_{\lambda h_{\lambda}}|^2}
\bigwedge_{k=1}^n\frac{\sqrt{-1}}{2\pi}dw_{\lambda k}\wedge d\bar w_{\lambda k},$$
where $\Phi$ is given by (\ref{1phi}) and $e_\lambda$ is a smooth positive function. 
Let $\{\phi_\lambda\}$ be a partition of unity subordinate to $\{U_\lambda\},$ then $\phi_\lambda e_\lambda$ is bounded on $N.$ Set 
$$\Phi_\lambda=\frac{\phi_\lambda e_\lambda}{|w_{\lambda1}|^2\cdots|w_{\lambda h_{\lambda}}|^2}
\bigwedge_{k=1}^n\frac{\sqrt{-1}}{2\pi}dw_{\lambda k}\wedge d\bar w_{\lambda k}.$$
Put $f_{\lambda k}=w_{\lambda k}\circ f$, then on  $f^{-1}(U_\lambda)$ we obtain
 \begin{equation}\label{56q}
   f^*\Phi_\lambda=
   \frac{\phi_{\lambda}\circ f\cdot e_\lambda\circ f}{|f_{\lambda1}|^2\cdots|f_{\lambda h_{\lambda}}|^2}
   \bigwedge_{k=1}^n\frac{\sqrt{-1}}{2\pi}df_{\lambda k}\wedge d\bar f_{\lambda k}.
 \end{equation}
Set 
$$f^*\Phi_\lambda\wedge\alpha^{m-n}=\xi_\lambda\alpha^m,$$
 then we have $\xi=\sum_\lambda\xi_\lambda.$ Again, set
\begin{equation}\label{gtou}
  f^*c_1(L,h)\wedge\alpha^{m-1}=\varrho\alpha^m.
\end{equation}
 Then 
\begin{equation}\label{d444}
  \varrho=\frac{1}{2m}e_{f^*\omega}.
\end{equation}
For each $\lambda$ and any $x\in f^{-1}(U_{\lambda}),$
take a  local holomorphic coordinate system $z$ around $x.$
Since $\phi_{\lambda}\circ f\cdot e_\lambda\circ f$ is bounded,  it is not very hard to see from (\ref{56q}) and (\ref{gtou}) that $\xi_\lambda$ is
 bounded from above by $P_\lambda,$
where $P_\lambda$ is a polynomial in
$$\varrho, \ \ g^{i\overline{j}}\frac{\partial f_{\lambda k}}{\partial z_i}\overline{\frac{\partial f_{\lambda k}}{\partial z_j}}\Big{/}|f_{\lambda k}|^2, \ \ 1\leq i, j\leq m, \ 1\leq k\leq n.$$
It yields that
$$\log^+\xi_\lambda\leq O\Big(\log^+\varrho+\sum_k\log^+\frac{\|\nabla_M f_{\lambda k}\|}{|f_{\lambda k}|}\Big)+O(1).$$
Thus, 
\begin{eqnarray}\label{fill}
   \log^+\xi  &\leq& O\Big(\log^+\varrho+\sum_{k,\lambda}\log^+\frac{\|\nabla_M f_{\lambda k}\|}{|f_{\lambda k}|}\Big)+O(1). \nonumber
\end{eqnarray}
By this with Theorem \ref{log} 
\begin{eqnarray*}
   & & \frac{1}{2}\mathbb E_o\big{[}\log\xi(X_{t})\big{]} \\
   &\leq& O\Bigg(\sum_{k,\lambda}\mathbb E_o\left[\log^+\frac{\|\nabla_M f_{\lambda k}\|}{|f_{\lambda k}|}(X_{t})\right]\Bigg)
   +O\Big{(}\mathbb E_o\left[\log^+\varrho(X_{t})\right]\Big{)}+O(1) \\
   &\leq& O\Bigg(\sum_{k,\lambda}\tilde{m}\left(t,\frac{\|\nabla_M f_{\lambda k}\|}{|f_{\lambda k}|}\right)\Bigg)
   +O\Big{(}\log^+\mathbb E_o\left[\varrho(X_{t})\right]\Big{)}+O(1) \\
   &\leq& O\Big{(}\sum_{k,\lambda}\log\tilde{T}_{f_{\lambda k}}(t, \omega_{FS})\Big{)}+O\Big{(}\log^+\mathbb E_o\left[\varrho(X_{t})\right]\Big{)}+O(1) \\
   &\leq& O\big{(}\log^+\tilde{T}_f(t,L)\big{)}+O\Big{(}\log^+\mathbb E_o\left[\varrho(X_{t})\right]\Big{)}+O(1).
  \end{eqnarray*}
Moreover, Lemma \ref{calculus} and (\ref{d444}) imply that 
\begin{eqnarray*}
\log^+\mathbb E_o\big{[}\varrho(X_{t})\big{]}&\leq& (1+\delta)\log^+\mathbb E_o\left[\int_0^{t}\varrho(X_s)ds\right] \\
   &=& \frac{(1+\delta)}{2m}\log^+\mathbb E_o\left[\int_0^{t}e_{f^*c_1(L,h)}(X_s)ds\right] \\
   &\leq& \frac{(1+\delta)}{m}\log^+\tilde{T}_f(t,L)+O(1).
\end{eqnarray*}
Combining the above with Lemma \ref{lemma}, we prove the theorem.

\subsection{Defect Relation}~

 Let $L_1, L_2$ be  holomorphic line bundles over  a  complex projective manifold $N.$ Define 
$$\underline{\left[\frac{c_1(L_2)}{c_1(L_1)}\right]}=\sup\big{\{}a\in\mathbb R: L_2>aL_1 \big{\}}, \ \ \overline{\left[\frac{c_1(L_2)}{c_1(L_1)}\right]}=\inf\big{\{}a\in\mathbb R: L_2<aL_1 \big{\}}.$$
It is clear that
\begin{equation}\label{yyy}
  \underline{\left[\frac{c_1(L_2)}{c_1(L_1)}\right]}\leq\liminf_{t\rightarrow\infty}\frac{\tilde{T}_f(t,L_2)}{\tilde{T}_f(t,L_1)}
\leq\limsup_{r\rightarrow\infty}\frac{\tilde{T}_f(t,L_2)}{\tilde{T}_f(t,L_1)}\leq\overline{\left[\frac{c_1(L_2)}{c_1(L_1)}\right]}.
\end{equation}

 Let  $M$ be a complete and stochastically complete K\"ahler manifold with $\dim_{\mathbb C}M\geq \dim_{\mathbb C}N,$ and let  $(L, h)$ be a positive Hermitian line bundle over $N.$ 
  For 
  $f:M\rightarrow N,$ a differentiably non-degenerate  holomorphic mapping such that $\tilde{T}_f(t,L)\rightarrow\infty$ as $t\rightarrow\infty,$
 we define the \emph{defect} of $f$ with respect to $D$   by
$$\tilde{\delta}_f(D)=1-\limsup_{t\rightarrow\infty}\frac{\tilde{N}_f(t,D)}{\tilde{T}_f(t,L)}.$$

\begin{theorem}[Defect relation]\label{defect} Assume the same conditions as in Theorem $\ref{nonpositive}$
 and $\tilde{T}_f(t,L)\rightarrow\infty$ as $t\rightarrow\infty.$
Then
$$\tilde{\delta}_f(D)\leq
\overline{\left[\frac{c_1(K^*_N)}{c_1(L)}\right]}-\underline{\left[\frac{\mathscr{R}_M}{f^*c_1(L)}\right]}.$$
\end{theorem}
\begin{proof} It follows from  Theorem \ref{nonpositive} that
$$1-\frac{\tilde{N}_f(t,D)}{\tilde{T}_f(t,L)}
   \leq \frac{\tilde{T}_f(t,K_N^*)}{\tilde{T}_f(t,L)}-\frac{\tilde{T}(t,\mathscr{R}_M)}{\tilde{T}_f(t,L)}.
   $$
Let $t\rightarrow\infty,$ then we have the theorem proved.
\end{proof}
\begin{cor}\label{ppoo} Assume the same conditions as in Theorem $\ref{defect}.$
 If  ${\rm{Ric}}_M\geq0,$  then
$$\tilde{\delta}_f(D)\leq
\overline{\left[\frac{c_1(K^*_N)}{c_1(L)}\right]}.$$
\end{cor}
\begin{proof} Since  ${\rm{Ric}}_M\geq0,$ then 
$$\underline{\left[\frac{\mathscr{R}_M}{f^*c_1(L)}\right]}\geq0.$$ 
This proves the corollary.
\end{proof}
\begin{cor}\label{} Let $D_j\in |L|$ for $1\leq  j\leq q$ such that $\sum_{j=1}^qD_j$ is of   simple normal crossing type.
 Assume the same conditions as in Theorem $\ref{defect}.$ If $s_M\geq0,$
then
$$\sum_{j=1}^q\tilde{\delta}_f(D_j)\leq
\frac{1}{q}\overline{\left[\frac{c_1(K^*_N)}{c_1(L)}\right]}.$$
\end{cor}

\begin{cor}\label{} Let $D_1,\cdots,D_q$ be hypersurfaces in $\mathbb P^n(\mathbb C)$ of degree $d_1,\cdots,d_q$ such that $\sum_{j=1}^qD_j$ is of   simple
normal crossing type. Let $f:M\rightarrow \mathbb P^n(\mathbb C)$  be a differentiably
non-degenerate holomorphic mapping  such  that $\tilde{T}_f(t,\omega_{FS})<\infty$ for $0<t<\infty.$  If $s_M\geq0$  and 
$$\mathbb E_o\left[\int_0^ts_M(X_s)ds\right]<\infty$$
for $0<t<\infty,$ then
$$\sum_{j=1}^qd_j\tilde\delta_f(D_j)\leq n+1.$$
\end{cor}
\begin{proof} Since $s_M\geq0$ implies that $\mathscr R_M\geq0,$ then it follows $R_M\geq0.$ Thus, it yields from Lemma \ref{condition} that $\tilde{T}_f(t,L)\rightarrow\infty$
as $t\rightarrow\infty$.  Furthermore, 
\begin{eqnarray*}
   0\leq \tilde{T}(t,\mathscr R_M)&=&-\frac{1}{4}\mathbb E_o\left[\int_0^{t}\Delta_M \log\det(g_{i\overline{j}})(X_s)ds\right] \\
&=& \frac{1}{2}\mathbb E_o\left[\int_0^ts_M(X_s)ds\right]<\infty
\end{eqnarray*}
and  
$$c_1(K^*_{\mathbb P^n(\mathbb C)})=(n+1)[\omega_{FS}], \ \ \  c_1(L_{D_j})=d_j[\omega_{FS}].$$  
Hence, we have the corollary proved.
\end{proof}

\begin{cor}\label{}  Let $a_1,\cdots,a_q$ be distinct  points in a compact Riemann surface $S$ of genus $g.$ Let $f:M\rightarrow S$
be a differentiably non-degenerate holomorphic mapping such  that $\tilde{T}_f(t,L_{a_1})<\infty$
 for $0<t<\infty.$  If $s_M\geq0$ and 
$$\mathbb E_o\left[\int_0^ts_M(X_s)ds\right]<\infty$$ for $0<t<\infty,$ then
$$\sum_{j=1}^q\tilde\delta(a_j)\leq 2-2g.$$
\end{cor}

If $M$ is parabolic, namely,  $X_t$ is recurrent, then  we obtain
\begin{theorem}\label{thm89} 
Let $L$ be a positive line bundle over a  complex projective manifold $N.$
   Let   $D\in|L|$ be  of  simple normal
crossing type. 
  Let $f:M\rightarrow N$ be a differentiably
non-degenerate holomorphic mapping.
If
\begin{equation}\label{con1}
  \int_Ms^-_M(x)dV(x)<\infty,
\end{equation}
then 

$(a)$ Let $R_M(x)\geq-cr^2(x)-c$ for a constant $c>0,$ where $R_M$ is defined by $(\ref{kappa}).$
If $f$ has finite energy, i.e., $$E(f):=\int_Me_{f^*c_1(L,h)}(x)dV(x)<\infty,$$ then
$$\tilde{\delta}_f(D)\leq
\overline{\left[\frac{c_1(K^*_N)}{c_1(L)}\right]}+\frac{\int_Ms^-_M(x)dV(x)}{E(f)}.$$

$(b)$ Let  $R_M(x)\geq-k(r(x))$ for a nondecreasing function $k\geq0$ such that
$k(r)/r^2\rightarrow0$ as $r\rightarrow\infty.$  If $(\ref{con2})$ is satisfied and
 $f$ has infinite energy, then
$$\tilde{\delta}_f(D)\leq
\overline{\left[\frac{c_1(K^*_N)}{c_1(L)}\right]}.$$
\end{theorem}
\begin{proof}  From Lemma \ref{condition}, we  note that $\tilde{T}_f(t,L)\rightarrow\infty$
as $t\rightarrow\infty.$
 Ricci curvature assumption implies that $M$ is stochastically complete, and parabolicity assumption implies that
 ratio ergodic theorem holds (see \cite{rev}). Using ratio ergodic theorem, we get
\begin{eqnarray*}
  \frac{\tilde{T}(t,\mathscr{R}_M)}{\tilde{T}_f(t,L)}&=& \frac{\mathbb E_o\left[\int_0^{t}s_M(X_s)ds\right]}
{\mathbb E_o\left[\int_0^{t}e_{f^*c_1(L,h)}(X_s)ds\right]}
\rightarrow
\frac{\int_M s_M(x)dV(x)}{\int_Me_{f^*c_1(L,h)}(x)dV(x)} \\
&=&\frac{\int_Ms_M(x)dV(x)}{E(f)}<\infty
\end{eqnarray*}
as $t\rightarrow\infty.$ Thus, $\tilde{T}(t,L)<\infty$ for $t<\infty$ and
$$ -\underline{\left[\frac{\mathscr{R}_M}{f^*c_1(L)}\right]}
\leq \frac{\int_Ms^-_M(x)dV(x)}{E(f)}.$$
By Theorem \ref{defect}, $(a)$ follows. For $(b),$ we first note that $\tilde{T}_f(t,L)$ makes sense since Lemma \ref{condition1}.  
 By  ratio ergodic theorem, we see that $(b)$ holds provided with  $E(f)=\infty.$
\end{proof}
If $M$ is  non-parabolic,  namely, $X_t$ is transient, then we obtain
\begin{theorem}\label{} Assume  that $(\ref{con1})$ holds and
$R_M(x)\geq-k(r(x))$ for a nondecreasing function $k\geq0$ satisfying 
$k(r)/r^2\rightarrow0$ as $r\rightarrow\infty.$
Let $L$ be a positive line bundle over a  complex projective manifold $N$
   and  $D\in|L|$ be of  simple normal
crossing type.
  Let $f:M\rightarrow N$ be a differentiably
non-degenerate holomorphic mapping satisfying $(\ref{con2})$ and  $\tilde{T}_f(t,L)\rightarrow\infty$
as $t\rightarrow\infty.$
 Then
$$\tilde{\delta}_f(D)\leq
\overline{\left[\frac{c_1(K^*_N)}{c_1(L)}\right]}.$$
\end{theorem}
\begin{proof} If $\mathscr R_M\geq0,$ the assertion  follows  from Theorem \ref{defect}.
If $\mathscr R_M<0,$ then  
$$|\tilde{T}(t,\mathscr{R}_M)|=\frac{1}{2}\mathbb E_o\left[\int_0^\infty s_M^-(X_t)dt\right].$$
The non-parabolicity of $M$ implies  that (see \cite{atsuji}, Theorem 22)
$$\mathbb E_o\left[\int_0^\infty R_M^-(X_t)dt\right]<\infty,$$
 By $s_M\geq m R_M,$ we see that  
$$|\tilde{T}(t,\mathscr{R}_M)|=\frac{1}{2}\mathbb E_o\left[\int_0^\infty s^{-}_M(X_t)dt\right]\leq
\frac{m}{2}\mathbb E_o\left[\int_0^\infty R^{-}_M(X_t)dt\right]<\infty.$$
Hence, $\tilde{T}(t,\mathscr{R}_M)$ is bounded. The theorem  follows from Theorem \ref{defect}.
\end{proof}

\noindent\emph{Proof of Theorem $\ref{zzzz}$.}
 ${\rm{Ric}}_M\geq0$  implies that $\tilde{T}_f(t,L)\rightarrow\infty$ 
as $t\rightarrow\infty$ since  Lemma \ref{condition},  and  the energy assumption means  that $\tilde{T}_f(t,L)<\infty$ for $t<\infty$ since Lemma \ref{condition1}.
By $\mathscr R_M\geq0,$  the theorem follows from Theorem \ref{defect}.

\subsection{The case when $M$ is an  algebraic manifold}~

In Section 4, we obtain an analogue  of  Nevanlinna theory on a wide class of K\"ahler manifolds. 
Sometimes, we are  more concerned  with   domain $M$ which is an algebraic manifold. Consider the algebraic manifold $M:=X\setminus S,$ where $X$ is a complex projective manifold and $S$ is a hypersurface of simple normal crossing type in $X.$ 
Note that $M$ is stochastically complete. 
Let $S=\sum_{j=1}^rS_j$ be a decomposition into irreducible components.  
Taking $\sigma\in H^0(M, L_{S})$ and $\sigma_j\in H^0(M, L_{S_j})$   satisfying  $\sigma=\sigma_1\otimes\cdots\otimes \sigma_r$ and $(\sigma_j)=S_j.$

Assume that  $(L_D, \tau)>0,$ i.e., the Chern form $c_1(L_D, \tau)>0$.  We  
consider situation for the following three typical complete K\"ahler metrics $\alpha$ on $M$ (see  \cite{at2}, pp. 1023), where the Second Main Theorem (Theorem  \ref{nonpositive}) still holds. 

${\rm{(I)}}$ Projective type: $\alpha=dd^c\log\|\sigma\|^{-2}.$ Under this metric, $M$ is parabolic, namely, the Brownian motion is recurrent. However, $M$ is not stochastically complete. Hence,  we cannot ensure  the desired property: $\tilde N_f(t, D)=0$ if $f$ omits $D.$ 

${\rm{(II)}}$ Euclidean type: $\alpha=dd^c\|\sigma\|^{-2}.$ Under this metric, the  Ricci curvature of $M$  is bounded and therefore $\tilde N_f(t, D)=0$ if $f$ omits $D.$ 
 Moreover, $M$ is non-parabolic for $\dim_{\mathbb C}\geq2,$ i.e., the Brownian motion is transient (see \cite{ga}).

${\rm{(III)}}$ Poincar\'e type: $\alpha=Cdd^c\log\|\sigma\|^{-2}-\sum_{j=1}^rdd^c\log(\log\|\sigma_j\|^2)^2.$ The metric was introduced by Cornalba-Griffiths \cite{cg}.  In   this case, 
we can consider a defect relation by choosing a suitable metric $\tau$ and a constant $C.$

\begin{lemma}\label{lemmax}  Assume that $L>0.$ Then there exist a constant $C>0$ and a Hermitian metric $\tau$ on $L$ such that $\alpha$ satisfies the following properties$:$

$(a)$ $\alpha$ is complete$;$

$(b)$ $M$ has finite volume with respect to $\alpha;$

$(c)$ ${\rm{Ric}}_M$ is bounded. More precisely, $-2\alpha\leq \mathscr R_M<0.$
\end{lemma}

In the above lemma, $(b)$ implies the parabolicity of $M;$ $(a)$ and $(c)$ ensures that $\tilde N_f(t, D)=0$ if $f$ omits $D.$ 

\begin{theorem}\label{thm90} Let  $f:M\rightarrow N$  be a differentiably non-degenerate
holomorphic mapping into a complex projective manifold $N$ with $\dim_{\mathbb C}N\leq\dim_{\mathbb C}M,$ where $M=X\setminus S$  is equipped with a complete K\"ahler metric $\alpha$ satisfying the properties of Lemma $\ref{lemmax}.$  Let $L$ be a positive line bundle over $N$. If $f$ satisfies $(\ref{con2}),$  then  
$$\tilde{\delta}_f(D)\leq
\overline{\left[\frac{c_1(K^*_N)}{c_1(L)}\right]}+\frac{4m\cdot {\rm{Vol}}(M)}{E(f)},$$
where $m=\dim_{\mathbb C}M.$
\end{theorem}
\begin{proof} By Lemma \ref{lemmax}, $-2\alpha\leq \mathscr R_M<0.$ It is therefore 
$$s^-_M=-2m\frac{\mathscr R\wedge\alpha^{m-1}}{\alpha^m}\leq 2m\frac{2\alpha\wedge\alpha^{m-1}}{\alpha^m}=4m.$$
According to  Theorem \ref{thm89}, we have the theorem proved.
\end{proof}

\begin{cor} Assume the same condition as in Theorem $\ref{thm90}.$ If $$\int_Me_{f^*c_1(L,h)}\alpha^m=\infty,$$
then 
$$\tilde{\delta}_f(D)\leq
\overline{\left[\frac{c_1(K^*_N)}{c_1(L)}\right]}.$$
\end{cor}

\vskip\baselineskip

\label{lastpage-01}

\vskip\baselineskip
\vskip\baselineskip

\end{document}